\documentclass[11pt,leqno,twoside]{article}
\usepackage{mathrsfs,amssymb,amsfonts,amsthm,amsmath,natbib,amsbsy}
\usepackage{dsfont,verbatim,fancyhdr}
\usepackage{enumerate}
\usepackage{graphicx,float}
\usepackage{cancel}
\usepackage[hang]{caption}
\usepackage[usenames,dvipsnames]{xcolor}
\usepackage[bookmarks,bookmarksnumbered,colorlinks=true,pdfstartview=FitV, linkcolor=Red,citecolor=blue!90!green,urlcolor=blue!90!green]{hyperref}
\usepackage{tikz}
\usepgflibrary{arrows}
\usepackage{bm}

\allowdisplaybreaks
\bibpunct{\textcolor{blue!90!green}{(}}{\textcolor{blue!90!green}{)}}{\textcolor{blue!90!green}{;}}{a}{\textcolor{blue!90!green}{,}}{\textcolor{blue!90!green}{;}}

\definecolor{iblue}{rgb}{0.1,0,0.75}
\definecolor{ired}{rgb}{0.9,0,0.1}

\def\N{\mathbb{N}}
\def\P{\mathbb{P}}
\def\X{\mathbb{X}}

\def\d{\mathrm{d}}
\renewcommand{\l}{\left}
\renewcommand{\r}{\right}
\newcommand{\cdist}{\stackrel{d}{\rightarrow}}

\newtheorem{theorem}{Theorem}[section]

\newtheorem{proposition}[theorem]{Proposition}

\newtheorem{definition1}[theorem]{Definition} 
 
\newtheorem{remark1}[theorem]{Remark} 

\usepackage{ifthen}

\long\def\symbolfootnote[#1]#2{\begingroup\def\thefootnote{\hspace*{-1mm}\fnsymbol{footnote}}\footnote[#1]{#2}\endgroup}

\fancypagestyle{plain}{
\fancyhead{}
}

\title{\bf \vspace{-2.5cm}
Clustering dynamics in a class of normalised generalised gamma dependent priors
\\[5mm]
}
\author{
\textsc{Matteo Ruggiero\footnote{Email: matteo.ruggiero@unito.it}}\\[1mm]
\emph{University of Torino and Collegio Carlo Alberto}\\[3mm]
\textsc{Matteo Sordello}\\[1mm]
\emph{University of Pennsylvania}
}
\date{\today}

\begin{document}
\thispagestyle{empty}

\maketitle

\begin{center}
\begin{minipage}{.75\textwidth}
\footnotesize\noindent
Normalised generalised gamma processes are random probability measures that induce nonparametric prior distributions widely used in Bayesian statistics, particularly for mixture modelling. 
We construct a class of dependent normalised generalised gamma priors induced by a stationary population model of Moran type, which exploits a generalised P\'olya urn scheme associated with the prior. 
We study the asymptotic scaling for the dynamics of the number of clusters in the sample, which in turn provides a dynamic measure of diversity in the underlying population. The limit is formalised to be a positive nonstationary diffusion process which falls outside well known families, with unbounded drift and an entrance boundary at the origin. We also introduce a new class of stationary positive diffusions, whose invariant measures are explicit and have power law tails, which approximate weakly the scaling limit.\\[-2mm]


\end{minipage}
\end{center}
\linespread{1.2}

\section{Introduction}
\label{intro}
A key aspect in Bayesian nonparametric inference, both from a methodological and a computational point of view, is the clustering of the observations. Regardless of whether these represent real quantities of interest or latent features used in intermediate levels of hierarchies, an often important inferential issue is the estimation of the number of components underlying the mixture. A Bayesian nonparametric mixture model typically takes the form
\begin{equation}\label{mixture}
f(y)=\int f(y\mid x)P(\d x)
\end{equation} 
where $f(\cdot| x)$ is a density function for every value of $x$, and the latent quantity $x$ is modelled through a random probability measure $P$. When $P$ is a Dirichlet process \citep{F73}, \eqref{mixture} is the Dirichlet process mixture model introduced by \cite{L84}, which is  to date the most popular Bayesian nonparametric approach.
The mixture \eqref{mixture} can be equivalently expressed in hierarchical form by writing
\begin{equation}\label{hierachical model}
\begin{split}
P\sim&\, \mathscr{P}\\
X_{i}\mid P\overset{iid}{\sim}&\, P\\
Y_{i}\mid X_{i}\overset{ind}{\sim}&\, f(\cdot\mid X_{i}),
\end{split}
\end{equation} 
where $\mathscr{P}$ is the prior induced by the random probability measure $P$ on the set of distributions over the latent space. If $\mathscr{P}$ selects almost surely discrete probability measures, as is the case when $P$ is a Dirichlet process \citep{B73}, then the latent variables $X_{i}$ will feature ties and can be used to cluster the observations $Y_{i}$ according to the kernel $f(\cdot| X_{i})$ from which they are generated. For this reason, the number of distinct values $K_{n}\le n$ in the exchangeable sequence $X_{1},\ldots,X_{n}$ is sometimes loosely referred to as the number of clusters. 

The implications on inference of the clustering structures induced by the Dirichlet process, which behave as logarithmic functions of the number of observations \citep{KH73}, have long been object of extensive investigation.
Recent advances in the field have drawn attention to different clustering behaviours, such as those induced by Pitman--Yor processes \citep{P95,PY97}, normalised inverse-gaussian processes \citep{LMP05} and normalised generalised gamma processes \citep{LMP07}. Despite the increased generality, these priors stand out for their tractability among the various generalisations of the Dirichlet process, and contrast with the latter by inducing clustering structures which behave as power functions of the number of observation. See \cite{LP10}.

Another recent, significant line of research in Bayesian nonparametrics aims at extending nonparametric priors in order to accommodate forms of dependence more general than exchangeability. 
 \cite{ME99,ME00} proposed a class of so-called  \textit{dependent processes} for modelling partially exchangeable sequences, where observations are exchangeable conditional on a given set of covariates, but not overall exchangeable. These are modelled through a collection of random probability measures with series representation of the form
\begin{equation}\label{series repr2}\nonumber
P_{z}=\sum_{i\ge1}p_{i,z}\delta_{x_{i,z}},
\end{equation} 
where the weights $p_{i,z}$ and/or the atoms $x_{i,z}$ depend on some covariate $z$, which can be multidimensional or possibly represents time. See \cite{HHMW} for a review and for recent developments. In particular, the discrete nature of these dependent priors and their wide applicability to concrete problems call for new understanding of how the clustering structure depends on the covariate, which is in turn induced by the type of dependence used for defining $p_{i,z}$ and $x_{i,z}$. 


%

In this paper we construct a new class of temporally dependent priors which are induced by a normalised generalised gamma population model, and investigate the scaling limit for the dynamics of the number of  groups or clusters. Other classes of dependent normalised random measures have been constructed recently in \cite{GKS13}, \cite{LNP14} and \cite{GL16}. 
{Here, by taking a different approach, the construction embeds normalised generalised gamma priors in a temporal environment. We define a dynamic population model of Moran type (see Section \ref{sec:model} for details), which can also be seen as the iteration of Gibbs sampling steps at the level of the latent variables in the hierarchy. We study the rescaling of the induced number of groups in the population and identify the limit to be a positive, nonstationary diffusion process which seems to fall outside well known classes. As the use of stationary components in dependent hierarchical model is often desirable, we formulate a weak approximation of the scaling limit by introducing a new family of stationary and positive diffusions whose invariant measure is explicit and exhibits a power law right tail.  }

{Scaling limits of the clustering dynamics for other classes of dependent models connected with Bayesian nonparametrics have been studied in  \cite{RWF13,R14} for the normalised inverse gaussian and the two-parameter Poisson--Dirichlet case, respectively. See also \cite{RW09a,RW09b,MRW11,MR15,PRS16} for different dependent models connected with diffusions processes. }

\section{Preliminaries on normalised generalised gamma priors}

Generalised gamma processes, introduced by \cite{B99}, are completely random measures with generalised gamma mean intensity, that is Levy processes with positive jumps and Levy measure on $[0,\infty)$ given by
$$\lambda(\d t) = \frac{e^{-\tau t} t^{-(1+\alpha)}}{\Gamma(1-\alpha)}\d t,$$
with $\alpha \in (0, 1)$ and $\tau \geq 0$. 
\cite{LMP07} exploited this construction for proposing a prior distribution for Bayesian nonparametric mixture modelling. This is obtained by normalising the jumps of a generalised gamma process via
$$p_i = \frac{J_i}{\sum_{k=1}^\infty J_k},$$
where $J_{i}$ are the jump sizes and $\sum_{k=1}^\infty J_k<\infty$ almost surely. The resulting random weights allow to define a discrete random probability measure by writing 
\begin{equation}\label{series repr}\nonumber
P=\sum\nolimits_{i\ge1}p_{i}\delta_{Z_{i}},
\end{equation} 
where $Z_{i}\overset{iid}{\sim}P_{0}$ and $P_{0}$ is a nonatomic probability measure on a  Polish space $\X$. The resulting normalised generalised gamma random probability measure induces a prior distribution on the space of discrete laws on $\X$, denoted here $\text{GG}(\beta,\alpha)$ for $\beta=\tau^{\alpha}/\alpha$. This can then be used at the top level of the hierarchy for Bayesian nonparametric modelling, replacing $\mathscr{P}$ with $\text{GG}(\beta,\alpha)$  in \eqref{hierachical model}. 

Denote by $K_n$ be the number of distinct values $X_{1}^{*},\ldots,X_{K_{n}}^{*}$ observed in a sample $X_{1},\ldots,X_{n}$ with $X_{i}\mid P\overset{iid}{\sim}P$. When $P\sim\text{GG}(\beta,\alpha)$, \cite{LMP07} showed that 
\begin{equation}\label{K distribution}
\P(K_{n}=k)=\frac{e^{\beta}\mathcal{G}(n,k,\alpha)}{\alpha\Gamma(n)}
\sum_{i=0}^{n-1}\binom{n-1}{i}(-1)^{i}\beta^{i/\alpha}\Gamma(k-i/\alpha;\beta),
\end{equation} 
where $\Gamma(a;x)$ is the incomplete gamma function, $\mathcal{G}(n,k,\alpha)$ is the generalised factorial coefficient (see \citealp{C05})
\begin{equation*}
\mathcal{G}(n,k,\alpha)=\frac{1}{k!}\sum_{j=0}^{k}(-1)^{j}\binom{k}{j}(-j\alpha)_{n}
\end{equation*} 
and $(a)_{n}=a(a+1)\cdots(a+n-1)$ is the increasing factorial. Furthermore, $K_{n}$ grows as $n^{\alpha}$ and
\begin{equation}\label{alpha-diversity}
\lim_{n \to \infty} \frac{K_n}{n^\alpha} = S \qquad\mbox{a.s.},
\end{equation} 
where $S$ is a random variable on $(0,\infty)$ with density
\begin{equation*}
g_{\beta,\alpha}(s)=\exp\{\beta-(\beta/s)^{1/\alpha}\}\alpha^{-1}s^{-1-1/\alpha}f_{\alpha}(s^{-1/\alpha}),
\end{equation*} 
and where $f_{\alpha}$ is the density of a positive stable random variable of index $\alpha$. Since $S$ summarises the asymptotic diversity in terms of number of groups which grows as a power function of $\alpha$, the partition associated with generalised gamma priors is said to have $\alpha$-diversity $S$. Cf.~\cite{P06}, Definition 3.10. 

Normalised generalised gamma priors belong to the larger class of Gibbs-type priors \citep{GP05,TPAMI}. These can be characterised, among other ways, in terms of the marginal law of the observations, which is given by a generalised P\'olya urn scheme. More specifically, conditionally on $K_n = k_n$, the predictive distribution for the observations associated with Gibbs-type models is given by the following generalised P\'olya urn scheme:
\begin{equation}\label{gen PU}
X_{n+1} | X_1, ..., X_n \sim g_0 (n, k_n)  P_{0}(\cdot) + g_1 (n, k_n) \sum_{j=1}^{k_n} (n_j - \alpha) \delta _{X_j^*}(\cdot).
\end{equation}
Here $P_{0}$ is as above, and the weights $g_{0}(n,k_{n}), g_{1}(n,k_{n})$, possibly dependent on other fixed parameters that characterise the specific model, satisfy
\begin{equation*}
g_0 (n, k_n)+(n-\alpha k_n)g_1 (n, k_n)=1
\end{equation*} 
for all $n\ge1$ and $1\le k_n\le n$. 
The interpretation of \eqref{gen PU} is that $g_0 (n, k_n)$ is the probability of sampling a previously unobserved value, 
{and the $g_1(n,k_n)(n_j-\alpha)$ is the probability of duplicating the distinct value $X_j^*$, thus enlarging the associated group by one unit.}
 \cite{LMP07} showed that in the normalised generalised gamma case we have
\begin{equation}\label{g0g1-GG}
\begin{aligned} 
g_0 (n, k_n) &= \frac{\alpha}{n} \frac{\sum_{i=0}^{n} \binom{n}{i} (-1)^i \beta^{i/\alpha} \Gamma(k_n + 1 - i/\alpha, \beta)}{\sum_{i=0}^{n-1} \binom{n-1}{i} (-1)^i \beta^{i/\alpha} \Gamma(k_n - i/\alpha, \beta)}, \\[3mm]
g_1 (n, k_n) &= \frac{1}{n} \frac{\sum_{i=0}^{n} \binom{n}{i} (-1)^i \beta^{i/\alpha} \Gamma(k_n - i/\alpha, \beta)}{\sum_{i=0}^{n-1} \binom{n-1}{i} (-1)^i \beta^{i/\alpha} \Gamma(k_n - i/\alpha, \beta)}.
\end{aligned}
\end{equation}
for $\alpha \in (0, 1)$ and $\beta$ as above.
Furthermore, the generalised gamma model is the only normalized completely random measure that is also of Gibbs type. See Proposition 2 in \cite{LPW08}.
{The Pitman--Yor process is also a member of the Gibbs family, in which case these quantities simplify to}
\begin{equation}\label{g0g1-PY}
g_0 (n, k_n) =\frac{\theta+\alpha k_n}{\theta+n},\quad \quad 
g_1 (n, k_n) = \frac{1}{\theta+n}
\end{equation}
for either $\alpha<0$ and $\theta=|\alpha|m$, $m\in \N$, or
\begin{equation}\label{parameters}
\alpha\in[0,1), \quad \quad \theta>-\alpha.
\end{equation} 

In this paper we aim at studying a dynamic version of the $\alpha$-diversity asymptotic result \eqref{alpha-diversity} for the generalised gamma model, after appropriately extending the distribution \eqref{K distribution} of the number of groups to a temporal framework, through the definition of a population model based on \eqref{gen PU}-\eqref{g0g1-GG}. To this end, we will make use of a recent result by \cite{AFNT15}, who extend a result contained in \cite{RWF13}. In particular, they derive the second order approximation of \eqref{g0g1-GG} to be
\begin{equation}\label{arbelapprox}
\begin{split}
g_0 (n, k_n) =&\, \frac{\alpha k_n}{n} + \frac{\beta}{k_n^{1/\alpha}} + o(n^{-1}), \\ 
g_1 (n, k_n) =&\, \frac 1n - \frac{\beta}{n k_n^{1/\alpha}} + o(n^{-2}),
\end{split}
\end{equation}
for $n\rightarrow \infty$, which allows to avoid, in view of an asymptotic study, a cumbersome computation with alternating sums and incomplete Gamma functions.

\section{A generalised gamma population model and its group dynamics}\label{sec:model}

\cite{RW09b} proposed a discrete construction for a class of two-parameter Poisson--Dirichlet diffusion models, introduced in \cite{P09}, based on the generalised P\'olya urn scheme \eqref{g0g1-PY}. Here we extend such approach for defining a stationary generalised gamma population model and derive the scaling limit for the dynamics of the number of groups (or species) in the population.  

Fix $n$, and let $X^{(n)}= (X_1, ... , X_n)$ be a sample from a generalised gamma model, with $X_1 \sim P_{0}$ and $X_{i} | X_1, ..., X_{i-1}$ as in (\ref{gen PU}) for $i=2,\ldots,n$. We update $X^{(n)}$ at discrete times by substituting a uniformly chosen coordinate of the vector with a replacement from its conditional distribution given the remaining observations. Given the exchangeability of the sample, and assuming we replace $X_{i}$, the new element has distribution
\begin{equation}\label{full conditional}
X_{i}' | X^{(n)}_{(-i)} \sim g_0 (n-1, k_{n,i})  P_{0}(\cdot) + g_1 (n-1, k_{n,i}) \sum_{j=1}^{k_{n,i}} (n_{j,i} - \alpha) \delta _{X_j^*}(\cdot)
\end{equation} 
where $X^{(n)}_{(-i)}=(X_1, ..., X_{i-1},X_{i+1},\ldots,X_{n})$ is the remaining sample after removing $X_{i}$, $k_{n,i}$ is the number of distinct values in $X^{(n)}_{(-i)}$ and $n_{j,i}$ is the cardinality of the $j$th cluster after removing $X_{i}$.
Thus $X_{i}'$ is of a new type with probability $g_0 (n-1, k_{n,i})$ or a copy of an existing type with probability $(n-1 - \alpha k_{n,i}) g_1 (n-1, k_{n,i})$.  In terms of the population model, copying an existing type is interpreted as a birth, whereby the offspring takes the parent type in a haploid population. New types are interpreted as births with mutation, where the mutant type does not depend on the parent type and is drawn from a pool of {infinitely many alleles}. {Removals are interpreted as deaths, which here keep the population size constant.} The resulting dynamics are those of a Moran model, which, together with Wright--Fisher models, are among the oldest approaches to mathematical population genetics. See \cite{E09} for background, and \cite{F10} for Moran and Wright--Fisher models applied to infinitely-many-alleles dynamics, with some connections to Bayesian nonparametrics. See also \cite{2pWF} for a recent Wright--Fisher construction of the two-parameter Poisson-Dirichlet diffusion. 

Denote the Markov chain resulting from the above described replacements by $X^{(n)}(\cdot)=\{X^{(n)} (m), m \in \N\}$, and define $K_{n}(\cdot)=\{ K_n (m) , m \in \N\}$ to be the process that tracks the number of distinct types in $X^{(n)}$.
Note that the dynamics of the Moran chain $X^{(n)}$ are equivalent to running a random scan Gibbs sampler \citep{SR93} on the joint distribution of a generalised gamma sample of size $n$. This implies the following. 

\begin{proposition}
Let $X^{(n)}(\cdot)=\{X^{(n)} (m), m \in \N\}$ be the Markov chain described  above, with transitions determined by replacing a randomly chosen coordinate $X_{i}$ with a sample from \eqref{full conditional}. Then $X^{(n)}(\cdot)$ is stationary. 
\end{proposition}
\begin{proof}
It follows by adapting the proof of Proposition 4.1 in \cite{RW09b}, which does not depend on the specific form of the urn weights, or equivalently by the stationarity of the Markov chain generated by a Gibbs sampler on $X^{(n)}$, given that \eqref{full conditional} are the full conditional distributions of the coordinates. 
\end{proof}

The stationary distribution of $X^{(n)}(\cdot)$ is clearly the joint law of an $n$-sized sample from \eqref{full conditional}. 
The previous result suggests that the present construction can be naturally embedded in broader Monte Carlo strategies where the distinct values of such observations represent the latent clusters for the data points. 

The transition probabilities of $K_n (m)$, denoted 
 $$p_n(k, k') = \P (K_n (m+1) = k' | K_n (m) = k), $$
can be easily derived from the dynamics of $X^{(n)}$. Denote by $M_{1,n}$ the number of types appearing only once in $X^{(n)}$. Then, the probability of a transition $k\mapsto k+1$ is given by the probability $1-M_{1,n}/n$ of not removing a group of size 1, times the probability $g_0 (n-1, k)$ of sampling a new type as a replacement. Similarly, the probability of a transition $k\mapsto k-1$ is given by the probability $M_{1,n}/n$ of removing a singleton, times the probability $(n - 1 - \alpha (k-1))g_1 (n-1, k-1)$ of duplicating an existing type as a replacement. Such transitions are not Markov, since $M_{1,n}$ carries more information than $K_{n}$. Following a similar approach to that in \cite{R14}, we can exploit the approximation $M_{1, n} \approx \alpha K_n$, deduced from Corollary 1 in \cite{LMP07}, to define
\begin{equation}\label{trprob0}
 p_n(k, k') = \left\{\begin{array}{ll}
\displaystyle  \left(1 - \frac{\alpha k}{n}\right)g_0 (n-1, k), &k' = k+1,\\[2mm]
\displaystyle 
\frac{\alpha k}{n} (n - 1 - \alpha (k-1))g_1 (n-1, k-1),\quad &k' = k-1,\\[2mm]
\displaystyle 1 - p(k, k+1) - p(k, k-1),  &k' = k,\\[1mm]
0,\quad \text{else};
\end{array}\right.
\end{equation}
(note that there is a misprint in {eq.~8 of} \cite{R14}, which should be as in \eqref{trprob0} with $g_{0},g_{1}$ as in \eqref{g0g1-PY}; i.e.,~dropping the small $o$ terms, with obvious modifications to the subsequent proof).
Using now \eqref{arbelapprox}, we can approximate \eqref{trprob0} with
 \begin{equation}{\label{trprob}}
 p_n(k, k') = \left\{\begin{array}{ll}
\displaystyle  \left(1 - \frac{\alpha k}{n}\right) \left(\frac{\alpha k}{n - 1} + \frac{\beta}{k^{1/\alpha}}+o(n^{-1})\right), \hspace{10mm}k' = k+1,\\[2mm]
\displaystyle 
\frac{\alpha k}{n}  (n -1 - \alpha (k-1)) \bigg(\frac{1}{n-1} - \frac{\beta}{(n-1) (k-1)^{1/\alpha}}+o(n^{-2})\bigg),\\[4mm]
\hspace{70mm}k' = k-1,\\[2mm]
\displaystyle 1 - p(k, k+1) - p(k, k-1),   \hspace{26mm}k' = k,
\end{array}\right.
\end{equation}
and 0 otherwise. Due to the approximation of $g_0 (n, k_n)$ and $g_1(n, k_n)$, non admissible values can arise for certain choices of parameters when $k_{n}$ is close to the boundary; hence the probabilities of $K_{n}$ stepping up or down are intended as $\min{(p_n(k, k+1), 1)}$ and $\max{(p_n(k, k-1), 0)}$ respectively.
Completed by the boundary conditions $p_n(1,0)=p_n(n,n+1)=0$, $K_{n}(m)$ with transitions \eqref{trprob} is clearly recurrent on $\{1,\ldots,n\}$. 

Define now $S(\cdot)=\{S (t) , t \geq 0 \}$ as the solution of the stochastic differential equation 
\begin{equation}\label{sde}
\d S (t) = \frac{\beta}{S(t)^{1/\alpha}} \d t + \sqrt{2 \alpha S(t)} \d B (t),
\end{equation}
where $B(t)$ is a standard Brownian motion. To the best of our knowledge, \eqref{sde} does not seem to belong to any well known class of diffusions. We will  first show that $S(t)$ above is a well defined diffusion process on $[0,\infty)$, it has an entrance boundary at $0$ and a natural boundary at $\infty$, and it is non stationary. An entrance boundary at the origin means that 0 can be the starting point of the process which instantly enters $(0,\infty)$ and never touches the origin again. A natural boundary at $\infty$ is instead attractive, but never reached. Then, we will show that $S(\cdot)$ is the scaling limit, as $n\rightarrow \infty$, of the above defined sequence of Markov chains after an appropriate space-time transformation.

\begin{proposition}{\label{boundaries}}
Let $S(\cdot)$ be the solution to \eqref{sde}. Then $S(\cdot)$ is a Feller process, has an entrance boundary at $0$ and a natural boundary at $\infty$, and it is non stationary.
\end{proposition}

\begin{proof}
Classical Feller theory leads to studying the boundary behaviour of the process by investigating some functionals of the drift and diffusion coefficients that characterise the process. Here we highlight the relevant quantities and refer to \cite{KT81}, Section 15.6, for further details (see also \citealp{E09}, Section 3). Define the scale function 
\begin{equation*}
Z(x) = \int^x z(y) dy , \quad \quad z(x) =\exp\l\{- \int^x \frac{2 \mu(y)}{\sigma ^2(y)} dy\r\} 
\end{equation*} 
and the speed measure
\begin{equation*}
M(x) = \int^x m(y) dy, \quad \quad m(x) = \frac{1}{\sigma ^2(x) z(x)}.
\end{equation*} 
A standard calculation leads to find
\begin{equation}\label{speed measure}
z(x)=e^{\beta x^{-\frac 1\alpha}}, \quad \quad 
m(x)=\frac{1}{2\alpha x} e^{-\beta x^{-\frac 1\alpha}}.
\end{equation} 
Lettin $Z(a)=\lim_{x\rightarrow a}Z(x)$ and similarly for $M$, {for $\alpha, \beta$ as in \eqref{g0g1-GG}}
it is easy to see that 
$Z(0) = Z(+\infty) = M(+\infty) = \infty$ and $M(0) < \infty$. Moreover, from
\begin{equation*}
\Sigma(x) = \int^x\l(\int_t^x z(y) dy\r) m(t) dt, \quad \quad N(x) = \int^x\l(\int_t^x m(y) dy\r) z(t) dt
\end{equation*} 
we deduce $\Sigma(0) = \Sigma(+\infty) = N(+\infty) = \infty$ and $N(0) < \infty$. The second assertion now follows from \cite{KT81}, Section 15.6. 

Let now $\hat{C}([0, \infty))$ be the Banach space of continuous functions on $[0, \infty)$ vanishing at infinity.
Let also 
\begin{equation}\label{generatore limite}
A f(s)= (\beta/s^{1/\alpha}) f'(s)+  \alpha s f''(s)
\end{equation} 
be the infinitesimal operator corresponding to \eqref{sde} and define
\begin{equation*}
\mathcal{D} = \{ f \in C([0, \infty)) \cap C^2((0, \infty)) : Af \in C([0, \infty)) \}.
\end{equation*} 
Corollary 8.1.1 in \cite{EK86}, together with the second assertion, implies that $\{(f,Af):\ f\in \mathcal{D} \cap \hat{C}([0, \infty))\}$ generates a Feller semigroup on $\hat{C}([0, \infty))$, which is the first statement.

The proof is completed by the fact that a stationary distribution must take the form 
\begin{equation}\label{stationary}
\psi (x) = m(x) [C_1 Z(x) + C_2]
\end{equation} 
and the above arguments imply that both constants must vanish.  
\end{proof}

Given the boundary properties shown in Proposition \ref{boundaries}, it follows that, without loss of generality, we can start $S(\cdot)$ from $(0,\infty)$ and take the latter as the state space of the processes.

The following Theorem, which extends Proposition 3 in \cite{LMP07}, shows that \eqref{sde} is the scaling limit of the sequence of  Markov chains $\{K_{n}\}_{n\ge1}$ with transitions \eqref{trprob}, in the sense that, as $n\rightarrow \infty$, the sequence of appropriately transformed chains converges in distribution to $S(\cdot)$. 
To this end, denote by $Z_{n}\cdist Z$ convergence in distribution, let $D_{A}(B)$ be the Skorohod space of right-continuous functions from $A$ to $B$ with left limits, and $C_{A}(B)$ its subspace of continuous functions endowed with the topology of uniform convergence. Let also $\lfloor\cdot\rfloor$ be the floor function. 

\begin{theorem}\label{thm: principale}
Let $K_{n}(\cdot)=\{K_n (m) , m \in \N \}$ be the Markov chain on $\N$ with transition probabilities as in (\ref{trprob}), where $\alpha \in (0, 1)$ and $\beta > 0$, and define $\tilde K_{n}(\cdot)=\{\tilde K_n (t) , t \geq 0 \}$ by
$$\tilde K_n (t) = \frac{K_n (\lfloor n^{1 + \alpha} t \rfloor)}{n^\alpha}.$$
Let also $S(\cdot)$ be as in \eqref{sde}.
If $\tilde K_n (0) \cdist S (0)$, then 
$$\tilde K_{n}(\cdot)\cdist S(\cdot)\quad \quad \mbox{in  } C_{[0, \infty)} ([0, \infty))$$
as $n \to \infty$.
\end{theorem}
\begin{proof}
Let $U_n$ be the semigroup operator induced by \eqref{trprob}. Writing $n$ and $k$ in place of $n-1$ and $k-1$ for brevity {given their asymptotic equivalence, we have}
\begin{align*}
U_n f(k) &= E[f(K_n(m + 1)) | K_n(m) = k ] \\
&= f(k+1) p_n(k, k+1) + f(k-1) p_n(k, k-1) + f(k) p_n(k, k) .
\end{align*}
Consider now the spatially rescaled variable and let $I$ denote the identity operator, leading to
\begin{align*}
(U_n - I) f\left(\frac{k_{n}}{n^\alpha}\right) = \ &E\l[f\l(\frac{K_n(m + 1)}{n^\alpha}\r) - f\l(\frac{K_{n}(m)}{n^\alpha}\r) \big| K_n(m) = k_{n} \r] \\
=&\left[f\left(\frac{k_{n}+1}{n^\alpha}\right) - f\left(\frac{k_{n}}{n^\alpha}\right)\right]\\
  &\times\left[\left(1 - \frac{\alpha k_{n}}{n}\right) \left(\frac{\alpha k_{n}}{n} + \frac{\beta}{k_{n}^{1/\alpha}}+o(n^{-1})\right)\right]  \\ 
 &+\left[f\left(\frac{k_{n}-1}{n^\alpha}\right) - f\left(\frac{k_{n}}{n^\alpha}\right)\right]\\
 &\times\left[ \frac{\alpha k_{n}}{n}  (n - \alpha k_{n}) \left(\frac 1n - \frac{\beta}{n k_{n}^{1/\alpha}}+o(n^{-2})\right)\right].
\end{align*}
A second order Taylor expansion, together with some standard computation, yields
$$(U_n - I) f(s_{n}) = n^{-1 - \alpha}\frac{\beta}{s_{n}^{1/\alpha}}f'(s_{n})  +  n^{-1 - \alpha} \alpha s_{n} f''(s_{n})+o(n^{-1-\alpha})$$
where $s_{n}=k_{n}/n^\alpha$.  Since $s_{n}\rightarrow s$ from \eqref{alpha-diversity}, it follows that 
$$n^{1 + \alpha} (U_n - I) f(s) \rightarrow  A f(s)$$
uniformly on $(0,\infty)$, for $f \in \mathcal{D}$, with $A$ as in \eqref{generatore limite}.
Theorem 1.6.5 in \cite{EK86} now implies that  
$$U_n\l(t/\varepsilon_{n} \r) f(s) \to U(t) f(s), \qquad \mbox{as } n \to \infty, \qquad\forall f(s) \in \hat{C}((0, \infty)),$$
where $\varepsilon_{n} = n^{-1-\alpha}$, where $U$ is the Feller semigroup operator corresponding to $A$. Then Theorem 4.2.6 of \cite{EK86} in turn implies that
$$\frac{K_n (\lfloor n^{1 + \alpha} t \rfloor)}{n^\alpha} \cdist S(t)$$
holds in $D_{[0, \infty)}((0, \infty))$, provided the weak convergence of the initial distributions holds on $(0,\infty)$. Since the $S(t)$ has null probability of touching the origin for all $t>0$, if the convergence of the initial distributions holds on $[0,\infty)$, then the weak convergence holds  in $D_{[0, \infty)}([0, \infty))$. 
 The full statement now follows from the fact that convergence in distribution on $D_{[0, \infty)}([0, \infty))$ to an object that belongs to $C_{[0, \infty)}([0, \infty))$ with  probability one, implies convergence in distribution on $C_{[0, \infty)}([0, \infty))$.
\end{proof}

\begin{figure}[t!]
\begin{center}
\includegraphics[width=.7\textwidth]{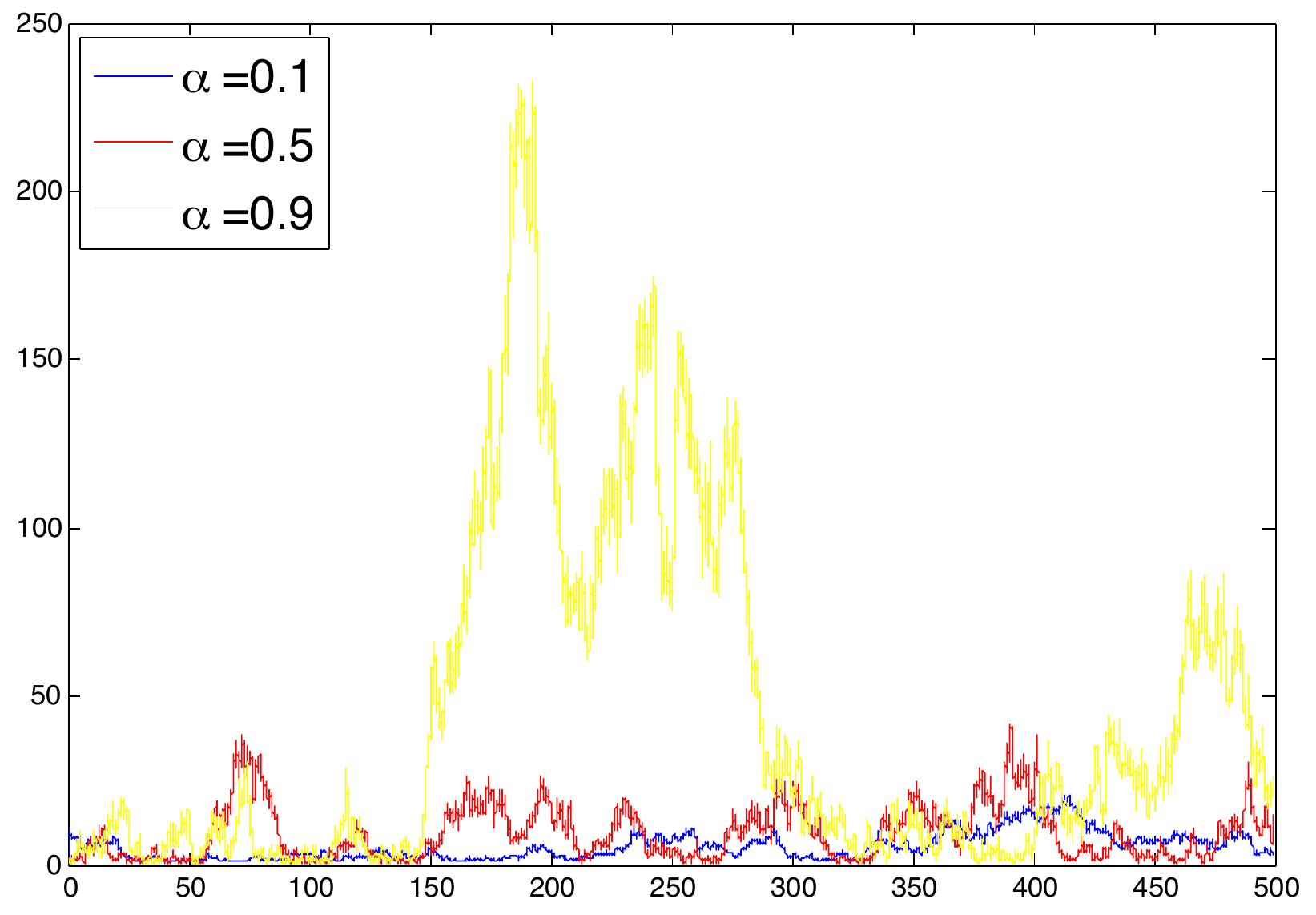}
\end{center}
\caption{
Some paths of \eqref{sde} for different values of $\alpha$.}
\label{fig: paths}
\end{figure}

The above Theorem states that the sequence of laws induced by the Markov chains $\tilde K_n (t)$ on the cadlag space of sample paths $D_{[0, \infty)}([0, \infty))$, converges weakly to the law induced by $S(\cdot)$ onto $C_{[0, \infty)}([0, \infty))$.
By analogy with \eqref{alpha-diversity}, the scaling limit $S(\cdot)$ in (\ref{sde}) can be interpreted as a dynamic measure of diversity in the generalised gamma population model constructed at the beginning of the present section. Figure \ref{fig: paths} shows some examples of sample paths of \eqref{sde} for different values of $\alpha$.


\section{Stationary approximations to the scaling limit}


Bayesian nonparametric inference in presence of temporally structured data usually tries to use stationary processes as building blocks of a broader model, as one typically has in mind a certain distributional structure for the marginal states and wants to make the latter depend on some covariate, such as time. Note that this approach is not particularly restrictive, as one can still model non stationary time series starting from stationary elements used for the construction in different hierarchical levels, in a similar spirit to hidden Markov models \citep{MR15}. It would then be desirable to have a stationary process describing the dynamics of the number of clusters. As this is not the case for the dynamics associated to generalised gamma clustering, as shown in Proposition \ref{boundaries}, we devise a weak approximation to the diffusion in Theorem \ref{thm: principale} such that any term of the approximating sequence is a stationary diffusion. This will provide stationary dynamics for the number of groups which are as close as desired to those induced by a generalised gamma population, with an explicit invariant measure. 

It is instructive to construct each term of the sequence of stationary diffusions from a continuous-time Markov chain, which highlights the underlying dynamics and allows a comparison with the results of the previous section. 
For any $\gamma > 0$, consider a continuous-time Markov chain $\{K_{n, \gamma} (t) , t \geq 0 \}$ on $\mathbb{N}$ with transition rates 
$$\lambda_1 = \frac{\alpha k^{1 + \gamma}}{n^{1+\alpha\gamma}} + \frac{\beta}{k^{1/\alpha}} + o (n^{-1})
=\frac{\alpha s^{1+\gamma}}{n^{1-\alpha}} + \frac{\beta/s^{1/\alpha}}{n} + o (n^{-1})
$$
from $k$ to $k+1$ and 
$$\lambda_2 =   \frac{\alpha k^{1 + \gamma}}{n^{1+\alpha\gamma}} + o (n^{-1})
=\frac{\alpha s^{1+\gamma}}{n^{1-\alpha}} + o (n^{-1})
$$
from $k$ to $k-1$. Here we are still assuming that $k_{n}/n^{\alpha}\rightarrow s_{n}$. 
The following result mimics Theorem \ref{thm: principale} and identifies the scaling limit of the sequence of Markov chains.

\begin{proposition}
Let $\{K_{n, \gamma} (t) , t \geq 0 \}$ be the above defined continuous time Markov chain with rates $\lambda_1$ and $\lambda_2$, and let $\{ \tilde{K}_{n, \gamma} (t) , t \geq 0 \}$ be defined as
$$\tilde{K}_{n, \gamma} (t) = \frac{K_{n, \gamma} ( n^{1 + \alpha} t )}{n^\alpha} .$$
Let $\{S_\gamma (t) , t \geq 0 \}$ be the diffusion process on $[0, \infty)$ driven by the stochastic differential equation 
\begin{equation}\label{sde2}
dS_\gamma (t) = \frac{\beta}{S_\gamma(t)^{1/\alpha}} dt + \sqrt{2 \alpha (S_\gamma(t))^{1+\gamma}} dB (t) .
\end{equation}
If $\tilde{K}_{n, \gamma} (0) \cdist S_\gamma (0)$ then 
$$\tilde{K}_{n, \gamma}(\cdot) \cdist S_\gamma (\cdot) \qquad\mbox{in  } C_{[0, \infty)} ([0, \infty))$$
as $n \to \infty$.
\end{proposition}
\begin{proof}
The proof proceeds along the same lines of that of Theorem \ref{thm: principale}. In particular the well definedness of the diffusion follows by the same argument for
\begin{equation*}
A_{\gamma}f(s)=\frac{\beta}{s^{1/\alpha}}f'(s) +  \alpha s^{1 + \gamma}f''(s).
\end{equation*} 
with $f\in \mathcal{D}_\gamma([0, \infty))$ and 
$$\mathcal{D}_\gamma([0, \infty)) =  \{ f \in C([0, \infty)) \cap C^2((0, \infty)) : A_\gamma f \in C([0, \infty)) \}.$$
Letting now $U_{n, \gamma}$ be the semigroup corresponding to the Markov chain $K_{n,\gamma}$, we have
\begin{align*}
(U_{n, \gamma} - I) f\left(\frac{k_{n}}{n^\alpha}\right) 
&= \frac{1}{n^\alpha} f'(s_{n}) (\lambda_1 - \lambda_2) + \frac{1}{2n^{2\alpha}} f''(s_{n}) (\lambda_1 + \lambda_2)+o(n^{-1-\alpha}) \\ 
&= n^{-1-\alpha}f'(s_{n}) \frac{\beta}{s_{n}^{1/\alpha}} +n^{-1-\alpha} f''(s_{n}) \alpha s_{n}^{1 + \gamma}+o(n^{-1-\alpha})  ,
\end{align*}
and the rest of the proof follows similarly.
\end{proof}

We conclude by showing that any process in the class $\{S_{\gamma}(\cdot)\}_{\gamma>0}$ is stationary, we identify the invariant measure and prove that for any sequence $\gamma_{\ell}\rightarrow 0$, the associated sequence of diffusions $\{S_{\gamma_{\ell}}(\cdot)\}_{\gamma_{\ell}}$ converges in distribution to $S(\cdot)$ in Theorem \ref{thm: principale}, as $\ell\rightarrow \infty$. For notational simplicity, we write $S_{\gamma}(\cdot)$ in place of $S_{\gamma_{\ell}}(\cdot)$. 

\begin{proposition}{\label{gammaboundaries}}
Let $S(\cdot)$ and $S_{\gamma}(\cdot) $ be as in \eqref{sde} and \eqref{sde2} respectively. 
For any $\gamma>0$, $S_\gamma (\cdot)$ is stationary with invariant measure 
\begin{equation}\label{stationary2}
\psi_\gamma (x) \propto  \frac{1}{2\alpha x^{1 + \gamma}} \exp\l\{- \frac{\beta}{1 + \alpha \gamma} x^{-\frac{1 + \alpha \gamma}{\alpha}}\r\},\quad \quad x>0,
\end{equation} 
with normalising constant $C = 2\beta (\frac{1 + \alpha \gamma}{\beta})^{\frac{1}{1+\alpha\gamma}}/\Gamma(\frac{\alpha\gamma}{1 + \alpha\gamma})$,
and has an entrance boundary at $0$ and a natural boundary at $\infty$. Moreover, as $\gamma \to 0$, $S_{\gamma}(\cdot)$ converges in distribution to $S(\cdot)$ on $C_{[0,\infty)}([0,\infty))$, provided the initial distributions converge.
\end{proposition}
\begin{proof}
Denote by $\mu_\gamma(\cdot)$ and $\sigma_{\gamma}(\cdot)$ the drift and diffusion coefficients in \eqref{sde2}. Then
\begin{align*}
z_\gamma(x) &= \exp\l\{- \int^x \frac{2 \mu_\gamma(y)}{\sigma_\gamma ^2(y)} dy\r\} 
= \exp\l\{\frac{\beta}{1 + \alpha \gamma} x^{-\frac{1 + \alpha \gamma}{\alpha}}\r\} 
\end{align*}
and 
\begin{equation*}
m_\gamma(x) = \frac{1}{\sigma_\gamma ^2(x) z_\gamma(x)} = \frac{1}{2\alpha x^{1 + \gamma}} \exp\l\{- \frac{\beta}{1 + \alpha \gamma} x^{-\frac{1 + \alpha \gamma}{\alpha}}\r\} .
\end{equation*} 
The function $z_\gamma(x)$ behaves essentially as $z(x)=z_{0}(x)$, so $Z_\gamma(0) = Z_\gamma(+\infty) = \infty$. The function $m_\gamma(x)$, instead, behaves like $m(x)=m_{0}(x)$ in a neighbourhood of $x = 0$, but  goes to $0$ as $1/x^{1+\gamma}$ for $x \to \infty$, so $M_\gamma(0) < \infty$ and $M_\gamma(+\infty) < \infty$. We immediately have $\Sigma_\gamma(0) = \Sigma_\gamma(+\infty) = \infty$. Moreover $N_\gamma(0) < \infty$ and $N_\gamma(+ \infty) = \infty$.
The boundary classification then again follows from \cite{KT81}, Section 15.6. 
From \eqref{stationary} we now find that $Z_\gamma(x) \equiv \infty$  implies $C_1 = 0$, whence $\psi_\gamma (x) \propto m_{\gamma}(x)$. 

Note now that the infinitesimal generators of $S_{\gamma}(\cdot)$ and $S(\cdot)$ satisfy
\begin{align*}
\l| A_{\gamma} f(s) - A f(s) \r| &= \left| \frac{\beta}{s^{1/\alpha}} f'(s) + \alpha s^{1 + \gamma} f''(s) - \left( \frac{\beta}{s^{1/\alpha}} f'(s) + \alpha s f''(s) \right)\right| \\
&=  \alpha | (s^{1 + \gamma} - s) f''(s) | \rightarrow 0 
\end{align*}
uniformly on $[0,\infty)$, for $f \in \mathcal{D}_{0}=\mathcal{D} \cap C^2([0, \infty))$, as $\gamma \to 0$. Now, $\mathcal{D}_{0}$  can be easily shown to be a core for $A$ (cf.~\cite{EK86}, Section 1.3), i.e.~its closure is such that $\overline{\mathcal{D}_{0}}=\mathcal{D}$ and $\overline{A\vert_{\mathcal{D}_{0}}}=A$ (here $\mathcal{D}_{0}$ and $\mathcal{D}$ differ for functions with one or two infinite derivatives at 0).
Theorems 1.6.1 and 4.2.5 in \cite{EK86} then yield
$$ \lim_{\gamma \to 0} U_{\gamma}(t) f = U(t) f \qquad \forall f \in  \hat{C}([0, \infty))$$
and 
$$S_{\gamma}(\cdot) \cdist S(\cdot) \qquad \mbox{  as  }  \gamma \to 0$$
on $D_{[0, \infty)}[0, \infty)$, provided the initial distributions converge. The rest of the argument is now analogous to the proof of Theorem \ref{thm: principale}.
\end{proof}

\begin{figure}[t!]
\begin{center}
\includegraphics[width=.7\textwidth]{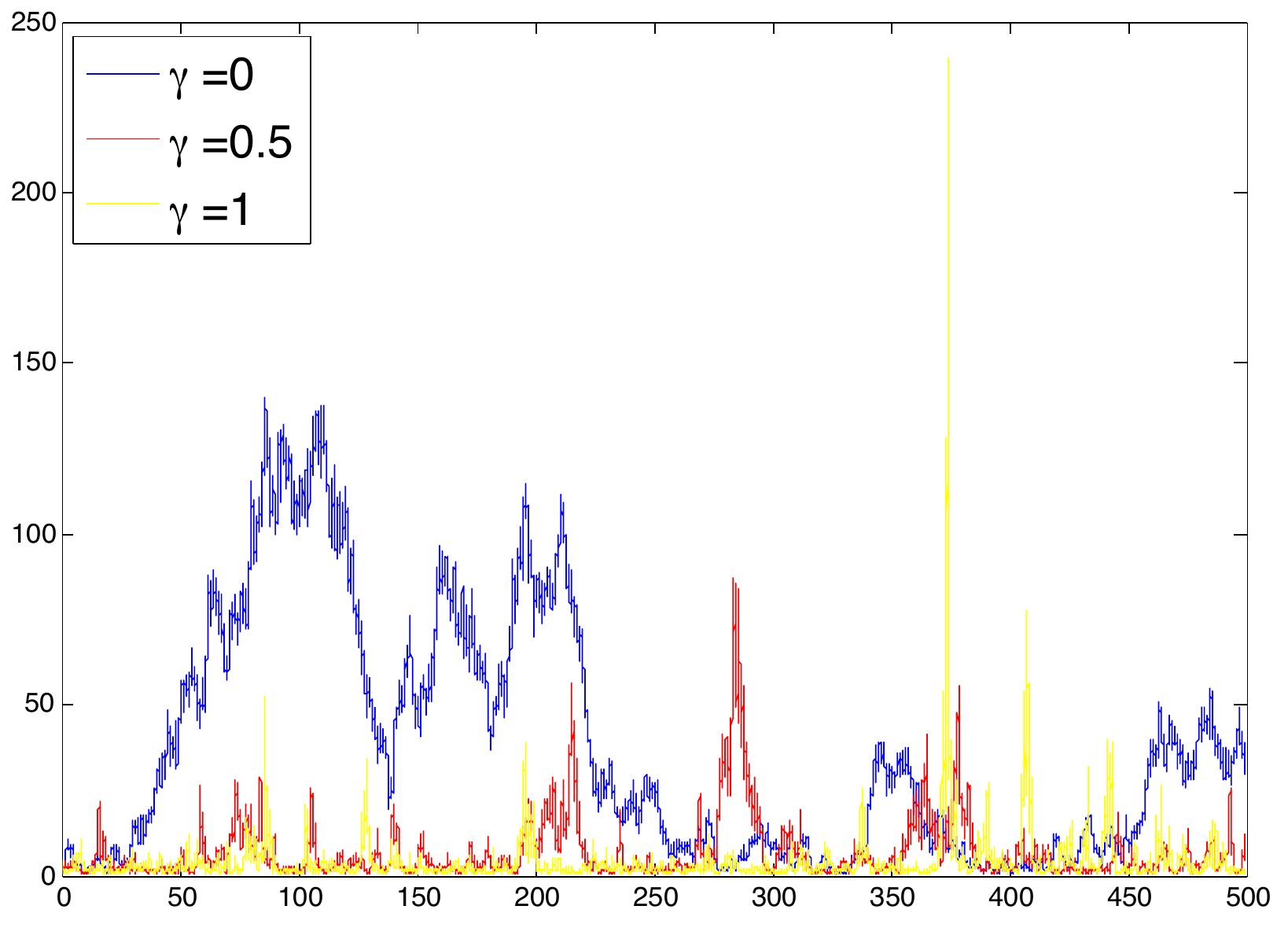}
\end{center}
\caption{Some paths of \eqref{sde2} for different values of $\gamma$.}\label{fig: paths2}
\end{figure}
\begin{figure}[t!]
\begin{center}
\includegraphics[width=.7\textwidth]{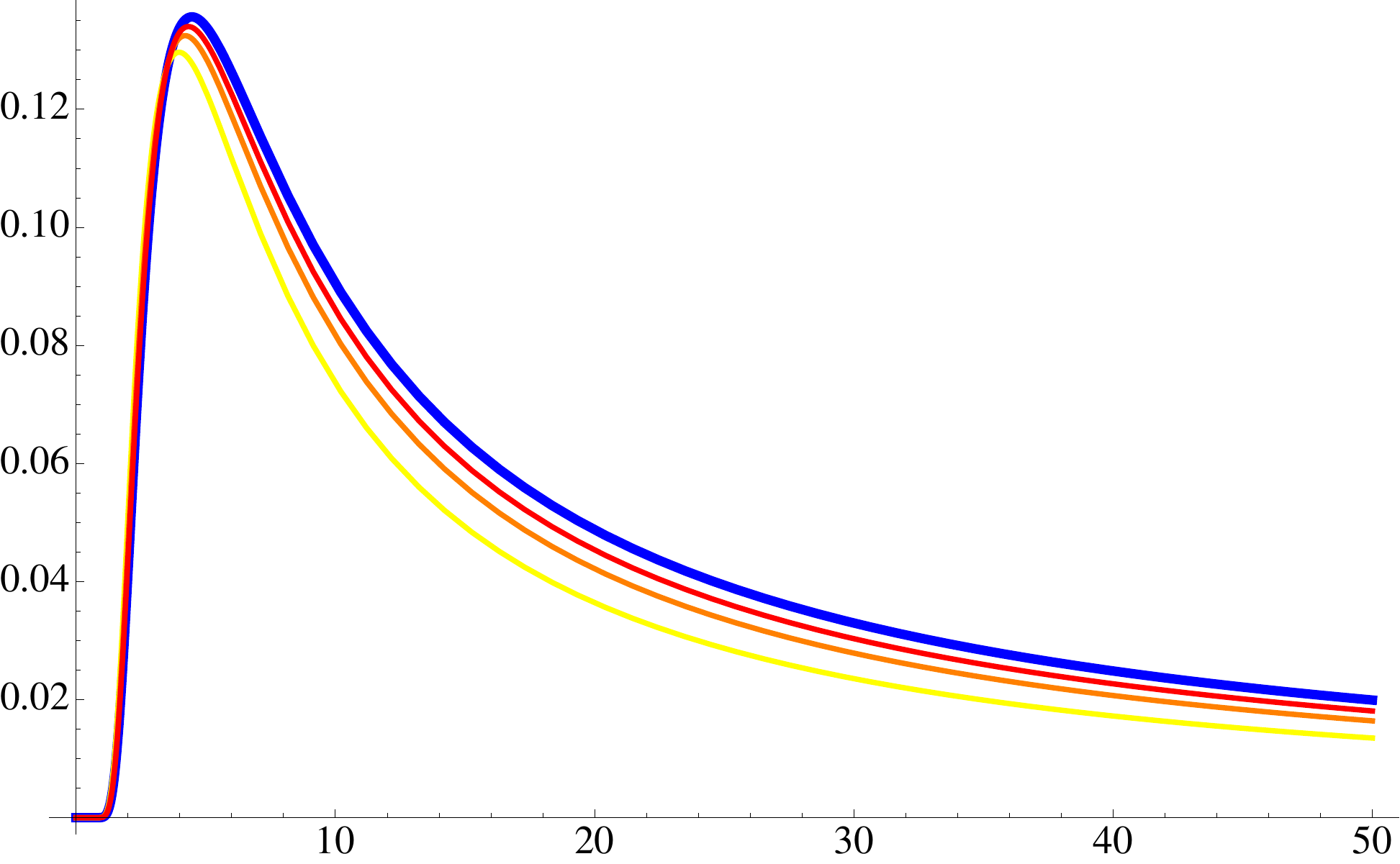}
\end{center}
\caption{Convergence of unnormalised stationary densities \eqref{stationary} for $\gamma=.1,.05,.025$ (yellow to red) to the speed measure \eqref{speed measure} of \eqref{sde} (blue), which does not integrate.}\label{fig: densities}
\end{figure}

Figure \ref{fig: paths2} shows the qualitative difference among sample paths of $S_{\gamma}(\cdot)$ for decreasing values of $\gamma$. Figure \ref{fig: densities} shows the convergence of the unnormalised stationary measures of $S_{\gamma}(\cdot)$  to the speed measure $m(x)$ of $S(\cdot)$ (blue curve), which does not integrate, for decreasing values of $\gamma$ (bottom to top); cf.~\eqref{speed measure} and \eqref{stationary2}. Here the stationary distribution of $S_{\gamma}(\cdot)$ has right tail decaying as $x^{-1-\gamma}$.


\section*{Acknowledgements}

The authors are grateful to two anonymous referees for helpful comments and to Pierpaolo De Blasi and Bertrand Lods for useful suggestions. The first author is supported by the European Research Council (ERC) through StG ``N- BNP'' 306406. This work was conducted while the second author was affiliated to the University of Torino and Collegio Carlo Alberto, Italy.



\end{document}